\documentclass[11pt,a4paper]{amsart}
\usepackage{amssymb,color}
\usepackage[english]{babel}          
\usepackage{verbatim}
\usepackage[T1]{fontenc}

\usepackage{epstopdf}

\usepackage{floatflt,graphicx}
\usepackage{color,xcolor}
\usepackage{enumerate}
\usepackage{animate}
\usepackage{amsmath}
\usepackage{amsthm}

\newcounter{minutes}
\divide\time by 60
\newcounter{hours}
\multiply\time by 60 \addtocounter{minutes}{-\time}
\title[Inequalities and bilipschitz conditions]{Inequalities and bilipschitz conditions for triangular ratio metric}

\author{Parisa Hariri}
\address{Department of Mathematics and Statistics,
  University of Turku, Turku, Finland}
\curraddr{}
\email{parisa.hariri@utu.fi}

\author{Matti Vuorinen}
\address{Department of Mathematics and Statistics,
  University of Turku, Turku, Finland}
\email{vuorinen@utu.fi}

\author{Xiaohui Zhang}
\address{Department of Physics and Mathematics,\
  University of Eastern Finland, Joensuu, Finland}
\email{xiaohui.zhang@uef.fi}

\keywords{triangular ratio metric, visual angle metric, distance ratio metric, bilipschitz condition, quasiconformal mappings}

\subjclass[2010]{51M10, 30C65}

\date{}

\dedicatory{}

\commby{}

\theoremstyle{plain}
\newtheorem{thm}[equation]{Theorem}
\newtheorem{cor}[equation]{Corollary}
\newtheorem{lem}[equation]{Lemma}

\newtheorem{prop}[equation]{Proposition}

\theoremstyle{definition}
\newtheorem{defn}[equation]{Definition}

\theoremstyle{remark}
\newtheorem{rem}[equation]{Remark}

\newtheorem{nonsec}[equation]{}

\numberwithin{equation}{section}

\newcommand{\beq}{\begin{equation}}
\newcommand{\eeq}{\end{equation}}

\newcommand{\bequu}{\begin{eqnarray*}}
\newcommand{\eequu}{\end{eqnarray*}}

\newcommand{\bequ}{\begin{eqnarray}}
\newcommand{\eequ}{\end{eqnarray}}

\newcommand{\B}{\mathbb{B}^2}
\newcommand{\R}{\mathbb{R}^2}

\newcommand{\Bn}{ {\mathbb{B}^n} }
\newcommand{\Rn}{ {\mathbb{R}^n} }

\newcommand{\sh}{\,\textnormal{sh}}
\newcommand{\ch}{\,\textnormal{ch}}
\renewcommand{\th}{\,\textnormal{th}}

\begin{document}
\def\thefootnote{}
\footnotetext{ \texttt{File:~\jobname .tex,
           printed: \number\year-\number\month-\number\day,
           \thehours.\ifnum\theminutes<10{0}\fi\theminutes}
} \makeatletter\def\thefootnote{\@arabic\c@footnote}\makeatother

\begin{abstract}
Let $G \subsetneq   \mathbb{R}^n$ be a domain and let $d_1$ and $d_2$ be two metrics on $G$. We compare the geometries defined by the two metrics to each other for several pairs of metrics. The metrics we study include the distance ratio metric, the triangular ratio metric and the visual angle metric. Finally we apply our results to study Lipschitz maps with respect to these metrics.
\end{abstract}

\maketitle

\section{Introduction}
\setcounter{equation}{0}
Several metrics have an important role in geometric function theory and in the study of quasiconformal maps in the plane and space \cite{g}, \cite{v1}, \cite{gp} and \cite{go}. One of the key topics studied is uniform continuity of quasiconformal mappings with respect to metrics. Many authors have proved that these maps are either Lipschitz or H\"older continuous with respect to hyperbolic type metrics \cite{go,v2,vu}.
J. Ferrand studied in \cite{f1} the reverse question: does Lipschitz continuity imply quasiconformality? A negative answer was given in \cite{fmv} in the case of a conformally invariant metric introduced by Ferrand \cite{f1}. Our goal here is to continue this research and to study similar questions for some other metrics, which, in our terminology, are of "hyperbolic type". While this term does not have a precise meaning, it refers to the fact
 that these metrics share some properties of the hyperbolic metric: the boundary has a strong influence on the value of the distance between points. In particular, we are interested in the visual angle metric introduced and studied recently in \cite{klvw} and the triangular ratio metric from \cite{klvw,chkv}. The triangular
ratio metric  is defined as follows for a domain
$G \subsetneq   \mathbb{R}^n$ and $x,y \in G\,$:
\begin{equation}\label{sm}
s_G(x,y)=\sup_{z\in \partial G}\frac{|x-y|}{|x-z|+|z-y|}\in [0,1]\,.
\end{equation}
The visual angle metric is defined by
\begin{equation}\label{v}
v_G(x,y)=\sup\{\measuredangle (x,z,y): z\in\partial{G} \},\quad x, y\in G\,,
\end{equation}
for domains $G\subsetneq \mathbb{R}^n,\, n\geq 2$, such that
$\partial{G}$ is not a proper subset of a line, see \cite[Lemma 2.8]{klvw}. Here
the notation $\measuredangle(x, z, y)$ means the angle in
the range $[0, \pi]$ between the segments $[x, z]$ and $[y, z]\,.$

This paper is divided into sections as follows.
In Section \ref{section 2} we give some preliminary results and prove various inequalities between the above metrics which will be applied later on. It is easy to see that there exist domains $G$ with isolated boundary points such that the metrics $s_G$ and $v_G$ are not comparable (see also \cite[Remark 2.18]{hvw}). Here we introduce in Section \ref{section 3} two conditions on domains $G$ under which $s_G$ and $v_G$ are comparable.
 The first condition applies to domains $G$ which satisfy that $\partial{G}$ is "locally uniformly nonlinear", see Theorem \ref{6.1},  whereas the second condition, similar to the so-called porosity condition, applies to domains satisfying "exterior ball condition", see Theorem \ref{myHdthm}. In Section \ref{section 4} we show, motivated in part by V\"ais\"al\"a's work \cite{v2}, that bilipschitz maps with respect to  the triangular ratio metric, distance ratio metric, and quasihyperbolic metric are quasiconformal. Finally, applying the results of Section \ref{section 2}, we prove that quasiregular mappings $f: \mathbb{B}^n \to \mathbb{B}^n $
 are H\"older continuous with respect to the metric $s_{\mathbb{B}^n}\,.$
\section{Preliminary results }\label{section 2}

We introduce some terminology and notation, following \cite{v1}. 
For  $x\in \Rn$  and  $r>0$  let
\begin{align*}
               B^n(x,r) & =\{\,z\in \Rn: |x-z|<r\,\},         \\
           S^{n-1}(x,r) & =\{\,z\in \Rn: |x-z|=r\,\}
\end{align*}
denote the ball and sphere, respectively, centered at  $x$  with radius
$r$. The abbreviations  $B^n(r)=B^n(0,r)$, $S^{n-1}(r)=S^{n-1}(0,r)$,
$B^n=B^n(1)$, $S^{n-1}=S^{n-1}(1)$  will be used frequently. The dimensions are sometimes omitted: $B(x,r)$, $S(x,r)$.
\begin{nonsec}{\bf Hyperbolic metric.}
The hyperbolic metrics $\rho_{\mathbb{H}^n}$ and $\rho_{\mathbb{B}^n}$ of the upper
half plane ${\mathbb{H}^n} = \{ (x_1,\ldots,x_n)\in {\mathbb{R}^n}:  x_n>0 \} $
and of the unit ball ${\mathbb{B}^n}= \{ z\in {\mathbb{R}^n}: |z|<1 \} $ can be defined as weighted metrics with the weight functions
 $w_{\mathbb{H}^n}(x)=1/{x_n}$ and
 $w_{\mathbb{B}^n}(x)=2/(1-|x|^2)\,,$ respectively.
Explicitly, by \cite[p.35]{b} we have for  $ x,y\in \mathbb{H}^n$
\begin{equation}\label{cro}
\ch{\rho_{\mathbb{H}^n}(x,y)}=1+\frac{|x-y|^2}{2x_ny_n},
\end{equation}
and by \cite[p.40]{b} for $x,y\in\mathbb{B}^n$
\begin{equation}\label{sro}
\sh{\frac{\rho_{\mathbb{B}^n}(x,y)}{2}}=\frac{|x-y|}{\sqrt{1-|x|^2}{\sqrt{1-|y|^2}}}\,.
\end{equation}
From \eqref{sro} we easily obtain
\begin{equation*}
\th{\frac{\rho_{\mathbb{B}^n}(x,y)}{2}}=\frac{|x-y|}{\sqrt{|x-y|^2+(1-|x|^2)(1-|y|^2)}}\,. 
\end{equation*}
\end{nonsec}
For both $\mathbb{B}^n$ and $\mathbb{H}^n$ one can define the hyperbolic metric using absolute ratios (cf. e.g. \cite[(2.21)]{vu}). Because of the M\"obius invariance of the absolute ratio we may thus define for every M\"obius transformation $h$ the hyperbolic metric in $h(\mathbb{B}^n)\,.$ This metric will be denoted by $\rho_{h(\mathbb{B}^n)}\,.$

\begin{nonsec}{\bf Distance ratio metric.}
For a proper open subset $G \subset {\mathbb R}^n\,$ and for all
$x,y\in G$, the  distance ratio
metric $j_G$ is defined as
$$
 j_G(x,y)=\log \left( 1+\frac{|x-y|}{\min \{d(x,\partial G),d(y, \partial G) \} } \right)\,.
$$
This metric was introduced by F.W. Gehring and B.G. Osgood
\cite{go} in a slightly different form and in the above form in \cite{vu0}. If confusion seems unlikely, then we also write $d(x)= d(x,\partial G)\,.$
In addition to $j_G\,,$ we also study the metric
$$
 j^*_G(x,y)= \th \frac{j_G(x,y)}{2} \,.
$$
Because $j_G$ is a metric,
it follows easily, see \cite[7.42(1)]{avv}, that $j^*_G$ is a metric, too.
Moreover, by \cite[Lemma 2.41(2)]{vu} and \cite[Lemma 7.56]{avv} if $G\in\{\mathbb{B}^n,\mathbb{H}^n\}\,,$ then
\beq\label{jrho}
j_G(x,y)\leq \rho_G(x,y)\leq 2 j_G(x,y)
\eeq
for all $x,y\in G\,.$
\end{nonsec}

\begin{nonsec}{\bf Quasihyperbolic metric.}
Let $G$ be a proper subdomain of ${\mathbb R}^n\,.$ For all $x,\,y\in G$, the quasihyperbolic metric $k_G$ is defined as
$$k_G(x,y)=\inf_{\gamma}\int_{\gamma}\frac{1}{d(z,\partial G)}|dz|\,,$$
where the infimum is taken over all rectifiable arcs $\gamma$ joining $x$ to $y$ in $G$ \cite{gp}.
From \cite[Lemma 2.1]{gp} it follows that
\beq\label{jk}
j_G(x,y)\leq k_G(x,y)
\eeq
for all $x,y\in G\,.$ It is easy to see that $k_{\mathbb{H}^n}\equiv \rho_{\mathbb{H}^n}$ and that for all $x,y\in\mathbb{B}^n$
\beq\label{rok}
\rho_{\mathbb{B}^n}(x,y)\leq 2 k_{\mathbb{B}^n}(x,y)\leq 2 \rho_{\mathbb{B}^n}(x,y)\,.
\eeq
\end{nonsec}

\begin{nonsec}{\bf Point pair function.}
We define for $x,y\in G\subsetneq\mathbb{R}^n$ the point pair function
$$p_G(x,y)=\frac{|x-y|}{\sqrt{|x-y|^2+4\,d(x)\,d(y)}}\,.$$
This point pair function was introduced in \cite{chkv} where it turned out to be a very
useful function in the study of the triangular ratio metric.
However, there are domains $G$ such that $p_G$ is not a metric.
\end{nonsec}

\begin{lem}\label{1.1}
Let $G$ be a proper subdomain of $\mathbb{R}^n\,.$ If $x, y\in G$, then
$$  j^*_G(x,y) =  \frac{|x-y|}{|x-y|+2 \min \{d(x),d(y) \}} $$
and
$$j^*_G(x,y)\leq s_G(x,y)\leq \frac{e^{j_G(x,y)}-1}{2}\,.$$
The first inequality is sharp for $G=\Rn\setminus\{0\}\,.$
\end{lem}

\begin{proof}
By symmetry we may assume that $d(x)\leq d(y)\,.$
For $x,y\in G$, let $z\in\partial{G}$ be a point satisfying $d(x)=|x-z|\,.$
 For the equality claim we see that
\allowdisplaybreaks\begin{align*}
\frac{|x-y|}{|x-y|+2d(x)} & =  \frac{|x-y|/d(x)}{|x-y|/d(x)+2}=\frac{e^{j_G(x,y)}-1}{e^{j_G(x,y)}+1}\\
& = \frac{e^{j_G(x,y)/2}-e^{-j_G(x,y)/2}}{e^{j_G(x,y)/2}+e^{-j_G(x,y)/2}}=j^*_G(x,y)\,.
\end{align*}
For the first inequality we observe that by the triangle inequality
$$
s_G(x,y)\geq \frac{|x-y|}{|x-z|+|z-y|}\geq \frac{|x-y|}{|x-y|+2d(x)}=j^*_G(x,y)\,.
$$
The sharpness of the first inequality when $G=\Rn\setminus\{0\}$, follows if we choose $x=1$, $y=t>1\,.$ Then $s_G(x,y)=\frac{t-1}{t+1}=j^*_G(x,y)\,.$
For the second inequality, note that
\allowdisplaybreaks\begin{align*}
s_G(x,y) & \leq  \frac{|x-y|}{d(x)+d(y)} \leq  \frac{|x-y|}{2 \sqrt{d(x)d(y)}}\\
& \leq  \frac{|x-y|}{2 d(x)}= \frac{e^{j_G(x,y)}-1}{2}\,.\qedhere
\end{align*}

\end{proof}

\begin{lem}\label{1.2}
Let $G$ be a proper subdomain of $\mathbb{R}^n\,.$ Then for all $x,y \in G$ we have
\[
s_G(x,y)\leq 2 j^*_G(x,y)\,.
\]
This inequality is sharp when the domain is $G=\Rn\setminus\{0\}\,.$
\end{lem}
\begin{proof}
We first consider the points $x,y\in G$ satisfying $e^{j_G(x,y)}\geq3\,.$ The definition of $j_G$
readily yields
$$2j^*_G(x,y)=\frac{2(e^{j_G(x,y)}-1)}{e^{j_G(x,y)}+1}\geq1\geq s_G(x,y)\,.$$

We next suppose that $e^{j_G(x,y)}<3\,.$ In this case, it is clear that
$$2j^*_G(x,y)\geq \frac{e^{j_G(x,y)}-1}{2}\,,$$
which together with Lemma \ref{1.1} implies the desired inequality.

The sharpness of the inequality can be easily verified by investigating the domain $G=\mathbb{R}^n \setminus \{ 0 \}\,.$ For any $x \in G$ selecting $y=-x$ gives $s_G(x,y) = 1$ and $j^*_G(x,y) = \frac12\,.$
\end{proof}

\begin{lem}\label{1.4} If $G$ is a proper subdomain of $\Rn$, then for all $x,y\in G$,
\[
j^*_G(x,y)\leq p_G(x,y) \leq \frac{w}{\sqrt{w^2+1}}\leq {\sqrt 2}j^*_G(x,y)\, ,
\]
with $ w= (e^{j_G(x,y)}-1)/2\,.$
Both bounds are sharp when the domain is $G=\Rn\setminus\{0\}\,.$
\end{lem}

\begin{proof}
Without loss of generality we may suppose that $d(x)\leq d(y)\,.$ Then by Lemma \ref{1.1} the first inequality is equivalent to
\[
\frac{|x-y|}{|x-y|+2d(x)}\leq \frac{|x-y|}{\sqrt{|x-y|^2+4 d(x)d(y)}}\,.
\]
This, in turn, follows easily from the inequality $d(y)\leq |x-y|+d(x)\,.$

For the second inequality observe that with $w=(e^{j_G(x,y)}-1)/2$
\[
p_G(x,y) = \frac{|x-y|}{2 d(x) \sqrt{(|x-y|/(2 d(x)))^2+d(y)/d(x)}}=\frac{w}{\sqrt{w^2 + d(y)/d(x)}}
\]
\[
\le \frac{w}{\sqrt{w^2 + 1}} \le \frac{1+w}{\sqrt{w^2 + 1}} j^*_G(x,y)\le \sqrt{2}  j^*_G(x,y) \,.
\]
To see the sharpness of the first inequality in $G=\Rn\setminus\{0\}$ if we choose $y=\frac1x$, $x>1$, then
\[
j^*_G(x,\frac1x)=\frac{x^2-1}{x^2-1+2}=p_G(x,\frac1x)\,.
\]
For the sharpness of the last inequality again in $G=\Rn\setminus\{0\}$, we choose $y=-x\,.$ Then
\[
p_G(x,-x)=\frac{1}{\sqrt{2}}, \quad j^*_G(x,-x)=\frac12 \, .\qedhere
\]
\end{proof}

\begin{prop} \label{prop29}
If $G$ is a bounded domain of $\mathbb{R}^n$, then for all $x,y\in G$
\[
j^*_G(x,y)\geq \frac{|x-y|}{d(G)}\,.
\]
\end{prop}
\begin{proof}
Fix $x,y\in G$ and a line $L$ through $x,y\,.$ Then there are points $x_1, y_1\in L\cap \partial G$ such that $x_1, x, y, y_1$ are in this order on $L$ and hence
\bequu
d(G)\geq |x_1-y_1|& = & |x_1-x|+|x-y|+|y-y_1|\\
&\geq & |x-y|+2\min \{d(x), d(y)\}\,.
\eequu
The proof follows from Lemma \ref{1.1}.
\end{proof}

\begin{lem}\label{1.3}
Let $G$ be a proper subdomain of $\Rn$, then for all $x,y\in G$,
\begin{enumerate}
\item
\[
\frac{1}{\sqrt{2}}\, p_G(x,y) \leq s_G(x,y)\leq 2 p_G(x,y)\,,
\]
\item
\[
s_G(x,y)\leq \frac{p_G(x,y)}{1- p_G(x,y)} \,.
\]
\end{enumerate}
\end{lem}

\begin{proof}
By symmetry we may suppose that $d(x)\le d(y)\,.$

(1) The lower bound follows from \cite[Lemma 3.4 (2)]{chkv}. For the upper bound observe that by Lemma \ref{1.2}
 \[
 s_G(x,y)\leq \frac{2|x-y|}{|x-y|+2d(x)}\leq \frac{2|x-y|}{\sqrt{|x-y|^2+4 d(x) d(y)}}=  p_G(x,y)\,,
 \]
 where the second inequality follows from the inequality $d(y)\leq d(x)+|x-y|\,.$

(2) The first inequality in Lemma \ref{1.4} can be written as
 $$   \frac{w}{1+w} \leq p_G(x,y)\,,\quad w= (e^{j_G(x,y)}-1)/2 \,.$$
This inequality implies (2) because $s_G(x,y) \leq w$ by Lemma \ref{1.1}.
\end{proof}
In \cite[3.23]{chkv}, it was proved that $\th \frac{\rho_{\Bn}(x,y)}{2}\leq 2 s_{\Bn}(x,y),$ for all $x, y\in \Bn\,.$ We next apply Lemma \ref{1.1} to improve this upper bound. Note that by \cite[(2.4), 3.4]{chkv} we have for all $x,y\in \mathbb{H}^n$
\begin{equation}{\label{mysH}}
s_{\mathbb{H}^n}(x,y)=p_{\mathbb{H}^n}(x,y)=\th \frac{\rho_{\mathbb{H}^n}(x,y)}{2}\, .
\end{equation}

\begin{lem}\label{lem2.12}
For $x, y\in \mathbb{B}^n$ we have
$$\th \frac{\rho_{\mathbb{B}^n}(x,y)}{4}\leq s_{\mathbb{B}^n}(x,y)\leq p_{\mathbb{B}^n}(x,y)\leq \th \frac{\rho_{\mathbb{B}^n}(x,y)}{2} \leq 2 \th \frac{\rho_{\mathbb{B}^n}(x,y)}{4}\,.$$
\end{lem}

\begin{proof}
For the first inequality, by Lemma \ref{1.1}, and \eqref{jrho}, we have
$$s_{\mathbb{B}^n}(x,y)\geq j^*_{\mathbb{B}^n}(x,y)= \th \frac{j_{\mathbb{B}^n}(x,y)}{2}\geq \th \frac{\rho_{\mathbb{B}^n}(x,y)}{4}\,.$$
The second and the third inequalities follows from \cite[Lemmas 3.4 (1) and 3.8]{chkv}
For the last inequality, by \cite[2.29 (1)]{vu},
\[
\th \frac{\rho_{\mathbb{B}^n}(x,y)}{4}=\frac{\th {\rho_{\mathbb{B}^n}(x,y)/2}}{1+\sqrt{1-\th^2 {\rho_{\mathbb{B}^n}(x,y)/2}}}.
\]
Therefore \[
2\th \frac{\rho_{\mathbb{B}^n}(x,y)}{4}\geq \th \frac{\rho_{\mathbb{B}^n}(x,y)}{2}\,,
\]
since $1+\sqrt{1-\th^2 {\rho_{\mathbb{B}^n}(x,y)/2}}\leq 2\,.$

\end{proof}



\begin{lem}
(1) Let $G$ be a proper subdomain of $\mathbb{R}^n\,.$ If $x, y\in G$, then
\[
\th{\frac{j_G(x,y)}{2}}\leq p_G(x,y)\leq \th{j_G(x,y)} \,.
\]

(2) If $G \subset  \mathbb{R}^n$ is a convex domain, $x,y \in G$ and $m= \min \{d(x), d(y)\}\,,$ then
\[
    {\rm th} (j_G(x,y)/2) \le s_G(x,y) \le \frac{|x-y|}{\sqrt{|x-y|^2 + 4 m^2}} \le {\rm th} j_G(x,y) \,.
\]
\end{lem}

\begin{proof}
(1) For the second inequality, by symmetry we may assume that $d(x)\leq d(y)\,.$ Writing $|x-y|=b$,
\[
p_G(x,y)=\frac{b}{\sqrt{b^2+4\,d(x)\,d(y)}}\leq \frac{b}{\sqrt{b^2+4\,d(x)^2}}\,,
\]
we have
\[
\th{j_G(x,y)}=\frac{e^{2j_G(x,y)}-1}{e^{2j_G(x,y)}+1}=\frac{\left(1+\frac{b}{d(x)}\right)^2-1}{\left(1+\frac{b}{d(x)}\right)^2+1}=
\frac{b^2+2b\,d(x)}{b^2+2b\,d(x)+2d(x)^2}\,.
\]
Denote $t=d(x)\,.$ Then the inequality
\[
p_G(x,y)\leq \th {j_G(x,y)}\,,
\]
is equivalent to
\[
\frac{b}{\sqrt{b^2+4\,t^2}}\leq \frac{b^2+2b\,t+2t^2}{b^2+2b\,t}\,,
\]
and the last inequality is equivalent to $4b^2\,t^3(2b+3t)\geq 0$, which is true, since $t=d(x)>0\,.$

The first inequality follows from Lemma \ref{1.4}.

(2) Fix $x,y \in G\,.$  Because $G$ is convex, by \cite[Lemma 3.4(1)]{chkv} and the proof of
(1) we have
$$
  s_G(x,y) \le p_G(x,y) \le \frac{|x-y|}{\sqrt{|x-y|^2 + 4 m^2}} \le {\rm th} j_G(x,y) \,.
$$
The first inequality in the claim follows from  Lemma \ref{1.4}.
\end{proof}


\begin{lem} \label{2.18}
For a convex domain $G\subsetneq\Rn$ and all $x, y\in G$ we have ${v_G(x,y)}\geq s_G(x,y)\geq j^*_G(x,y)\,.$
\end{lem}

\begin{proof}
By \cite[Lemma 2.16]{hvw} $s_G(x,y)\leq v_G(x,y)$, so the result follows directly from Lemma \ref{1.1}.
\end{proof}

\begin{thm} \label{svconvex}
For a convex domain $G\subsetneq\Rn$ and all $x, y\in G$ we have
\begin{enumerate}
\item
\[
s_G(x,y)\leq \sqrt{2}j^*_G(x,y)\,,
\]
and
\item
\[
v_G(x,y)\geq \frac{1}{\sqrt{2}} p_G(x,y)\,.
\]
\end{enumerate}
\end{thm}
\begin{proof}
(1) This inequality follows from Lemma \ref{1.4} and \cite[Lemma 3.4]{chkv}.\\
(2) By Lemmas \ref{1.4} and \ref{2.18} we have
\[
v_G(x,y)\ge j^*_G(x,y)\ge \frac{1}{\sqrt{2}}p_G(x,y)\,.\qedhere
\]
\end{proof}

The next theorem shows that the constant $1/\sqrt{2}$ in Theorem \ref{svconvex} (2) can be improved for the case of a half space or a ball to be $1\,.$
The sharp constant in the case of a convex domain  will be given in Remark \ref{sharpc}.

\begin{thm}\label{2.13}
Let $G$ be a half space or a ball in the Euclidean space $\mathbb{R}^n\,.$ Then for all $x, y\in G$
$$v_{G}(x,y)\geq p_{G}(x,y)\,.$$
\end{thm}

\begin{proof}
Since both the visual angle metric $v_G$ and the point pair function $p_G$ are invariant under the similarities of the domain $G$, we may assume that the domain $G$ is the upper half space $\mathbb{H}^n$ or the unit ball $\mathbb{B}^n\,.$
We first consider the case of $G=\mathbb{H}^n\,.$ By the left-hand side inequality of \cite[Theorem 3.19]{klvw} and the well-known Shafer inequality $\arctan t\geq 3t/(1+2\sqrt{1+t^2})$ for $t>0$ (see \cite{s} or \cite{avz}), we have that
\allowdisplaybreaks\begin{align*}
v_{\mathbb{H}^n}(x,y)&\geq  \arctan\left(\sh\frac{\rho_{\mathbb{H}^n}(x,y)}{2}\right)\\
                     &\geq \frac{3\sh(\rho_{\mathbb{H}^n}(x,y)/2)}{1+2\sqrt{1+\sh^2(\rho_{\mathbb{H}^n}(x,y)/2)}}=:A\,.
\end{align*}
By \eqref{cro}, we have that
$$\sh\left(\frac{\rho_{\mathbb{H}^n}(x,y)}{2}\right)=\sqrt{\frac{\ch \rho_{\mathbb{H}^n}(x,y)-1}{2}}=\frac{|x-y|}{2\sqrt{d(x)d(y)}}\,,$$
and hence
\allowdisplaybreaks\begin{align*}
A&= \frac{3|x-y|}{\sqrt{4d(x)d(y)}+2\sqrt{|x-y|^2+4d(x)d(y)}}\\
&\geq \frac{|x-y|}{\sqrt{|x-y|^2+4d(x)d(y)}}=p_G(x,y)\,.
\end{align*}

For the case of $G=\mathbb{B}^n$, we use the inequality $$\arctan\left(\sh \frac{\rho_{\mathbb{B}^n}(x,y)}{2}\right)\leq v_{\mathbb{B}^n}(x,y)\,,$$ see \cite[Theorem 3.11]{klvw}. The same argument as in the case of the upper half space gives the proof for $v_{\mathbb{B}^n}(x,y)\geq p_{\mathbb{B}^n}(x,y)\,.$
\end{proof}

\begin{rem} \label{sharpc}
For a general convex domain $G\subset \Rn$, the inequality $v_G\geq p_G$ may not hold. Consider the strip domain $S=\{(x,y)\in\mathbb{R}^2:\, -\infty<x<\infty, -1<y<1\}$ and two points $a=(0,t)$, $b=(0,-t)$ for $0<t<1\,.$ Then it is easy to see that
$$p_S(a,b)=\frac{t}{\sqrt{t^2+(1-t)^2}},\qquad {\rm and}\quad v_S(a,b)=\arcsin t\,.$$
We see that
$$C:=\inf\limits_{t\in(0,1)}\frac{v_S(a,b)}{p_S(a,b)}=0.73707\dots>1/\sqrt 2=0.707107\dots$$

Actually, one can prove that, in general, for a convex domain $G$ we have that
\begin{equation}\label{vpconvex}
v_G\geq C p_G,\quad C=0.73707\dots
\end{equation}
Let $t=e^{j_G(x,y)}-1\,.$  To this end we apply the inequality \cite[Theorem 4.1]{vw} which says that for a convex domain $G$ and $x,y \in G\,,$
\[
v_G(x,y)\geq \arcsin{\frac{t}{t+2}}\,.
\]
On the other hand, it is easy to see that
$$
p_G(x,y)\leq \frac{t}{\sqrt{t^2+4}}\,.
$$
Hence we have that
$$
\frac{v_G(x,y)}{p_G(x,y)}\geq\frac{\arcsin(t/(t+2))}{t/\sqrt{t^2+4}}=\frac{\arcsin s}{s/\sqrt{s^2+(1-s)^2}}\geq C\,,
$$
where $s=t/(t+2)\,.$ The above example of the strip domain shows that the constant $C$ is best possible. Thus the inequality \eqref{vpconvex} improves Theorem \ref{svconvex} (2).
\end{rem}

\begin{lem}\label{kz}    
Let $G$ be a proper subdomain of $\Rn$, $z\in G$, and let $\lambda\in (0,1)\,.$ Then for all $x, y\in B(z,\lambda d(z))$
\[
k_{B(z, d(z))}(x,y)\leq \frac{1+\lambda}{1-\lambda}k_G(x,y)\,.
\]

\end{lem}

\begin{proof}
By the definition $k_G(x,y)=\int_{J_G}\frac{|du|}{d(u,\partial{G})}$, where $J_G$ is the geodesic segment of the metric $k_G$ joining $x$ and $y$ in $G\,.$ Because $x, y\in B(z,\lambda d(z))$ it follows from \cite[Theorem 2.2]{m} that
$J_G\subset B(z,\lambda d(z))$ and hence for all $u\in J_G$, $d(u,\partial{G})\leq (1+\lambda)d(z)$ and further
\allowdisplaybreaks\begin{align*}
k_{B(z, d(z))}(x,y) &\leq \int_{J_G}\frac{|du|}{d(u,S^{n-1}(z, d(z)))}=\int_{J_G}\frac{|du|}{d(z)-|u-z|}\\
&\leq \int_{J_G}\frac{|du|}{(1-\lambda)d(z)}
\leq \int_{J_G}\frac{|du|}{\frac{1-\lambda}{1+\lambda}d(u,\partial{G})}
= \frac{1+\lambda}{1-\lambda}k_{G}(x,y)\,.
\end{align*}

\end{proof}

\begin{thm}        
Let $G\subset\Rn$,  $x, y\in G$, and  $\lambda\in (0,1)\,.$ Then $$s_{G}(x,y)\leq c\th\left(\frac{1+\lambda}{1-\lambda}k_G(x,y)\right)\, ; \quad c=\frac{1}{\th\left(\frac{1+\lambda}{1-\lambda}\log(1+\lambda)\right)}\,.$$
\end{thm}

\begin{proof}
We divide the proof into two cases:

Case1: $x, y\in B(z,\lambda d(z))$ for some $z\in G\,.$

By domain monotonicity, Lemma \ref{lem2.12}, \eqref{jrho}, \cite[3.4]{vu} and Lemma \ref{kz}
\allowdisplaybreaks\begin{align*}
s_G(x,y)&\leq s_{B(z, d(z))}(x,y) \leq \th\left(\frac{\rho_{B(z, d(z))}(x,y)}{2}\right)\\
&\leq   \th\left(j_{B(z, d(z))}(x,y)\right)
\leq  \th\left(k_{B(z, d(z))}(x,y)\right)
\leq  \th\left(\frac{1+\lambda}{1-\lambda}k_{G}(x,y)\right).
\end{align*}

Case 2: Case 1 is not true. Then choosing $z=x$ we see that $y\not\in B(x,\lambda d(x))$ and hence by \eqref{jk}, $k_G(x,y)\geq \log(1+\lambda)$ because $|x-y|\ge \lambda \min \{d(x),d(y)\}$ and
\[
s_G(x,y)\leq c \th\left(\frac{1+\lambda}{1-\lambda}\log(1+\lambda)\right)\leq c\th\left(\frac{1+\lambda}{1-\lambda}k_G(x,y)\right)
\]
holds if $c= \frac{1}{\th\left(\frac{1+\lambda}{1-\lambda}\log(1+\lambda)\right)}\,.$
\end{proof}

\begin{rem}
A uniform domain $G\subset\Rn$ is a domain with the following comparison property between the quasihyperbolic metric and the distance ratio metric: there exists a constant $C>1$ such that, for all $x,y\in G$,
$$j_G(x,y)\leq k_G(x,y) \leq C j_G(x,y)\,.$$
Recall that the lower bound holds for all domains $G$ by \eqref{jk}.
In \cite[Theorem 2]{go} a similar characterization of uniform domains was given but with the expression $a j_G(x,y)+b$ on the right hand side. It was pointed out in \cite[2.50 (2)]{vu0} that the pair of constants $(a,b)$ can be replaced by $(c,0)$ where $c$ only depends on $(a,b)\,.$
Hence, this comparison property and the above results yield numerous new inequalities between the quasihyperbolic metric and the triangular ratio metric or the visual angle metric in uniform domains. The class of uniform domains is very wide: for instance quasidisks in $\mathbb{R}^2$ are
such domains  \cite{gh}.

\end{rem}

\section{Comparison results between triangular ratio metric and visual angle metric}\label{section 3}

We introduce in this section two conditions on domains $G$ for which $s_G$ and $v_G$ are comparable.
 The first condition applies to domains $G$ which satisfy that $\partial{G}$ is "locally uniformly nonlinear", see Theorem \ref{6.1},  whereas the second condition applies to domains satisfying "exterior ball condition".

Very recently, after the submission of this paper we found another proof of Theorem \ref{3.1}. See \cite[Lemma 2.11]{hvw}.

\medskip
\begin{thm}\label{3.1}
If $G\subset\Rn$ is a domain, then for all $x, y\in G\,,$
$$s_G(x,y)\geq \sin{\frac{v_G(x,y)}{2}}\,.$$
\end{thm}

\begin{proof}

Let $w_0\in\partial{G}$ be a point such that $v_G(x,y)=\measuredangle (x,w_0,y)\,.$ Let $E$ be the envelope of the pair $(x,y)$ which defines $v_G(x,y)$ (see \cite[2.9]{klvw}).
Clearly,
\allowdisplaybreaks\begin{align*}
s_G(x,y)&\geq \frac{|x-y|}{|x-w_0|+|w_0-y|}\\
        &\geq \inf_{w\in\partial{E}}\frac{|x-y|}{|x-w|+|w-y|}\,.
\end{align*}
We need to get the maximum of $|x-w|+|w-y|$ when $w\in\partial{E}\,.$ It is easy to check that the radius of the boundary circular arcs of the envelope $E$ is $R=\frac{|x-y|}{2\sin{v_G(x,y)}}\,.$ For $w\in\partial{E}\,,$ let $\theta$ be the central angle formed by the points $y\,,$ $w$ and the center. We see that
\allowdisplaybreaks\begin{align*}
|x-w|+|w-y| &= 2R \sin{\frac{\theta}{2}}+2R \cos\left(v_G(x,y)-\frac{\pi-\theta}{2}\right)\\
&= 2R \sin{\frac{\theta}{2}}+2R \sin\left(v_G(x,y)+\frac{\theta}{2}\right)\\
&\equiv  f(\theta)\,,
\end{align*}
and
$$
\max (f(\theta)) = f(\pi-v_G(x,y))=4R\cos{\frac{v_G(x,y)}{2}}\,.
$$
Therefore,
\allowdisplaybreaks\begin{align*}
s_G(x,y) &\geq  \frac{|x-y|}{4R\cos{\frac{v_G(x,y)}{2}}}\\
&= \frac{|x-y|}{\frac{4|x-y|}{2\sin{v_G(x,y)}}\cdot\cos{\frac{v_G(x,y)}{2}}}\\
&= \sin{\frac{v_G(x,y)}{2}}\,.\qedhere
\end{align*}
\end{proof}

\begin{figure}[h]

\begin{center}
\includegraphics[width=7.5cm]{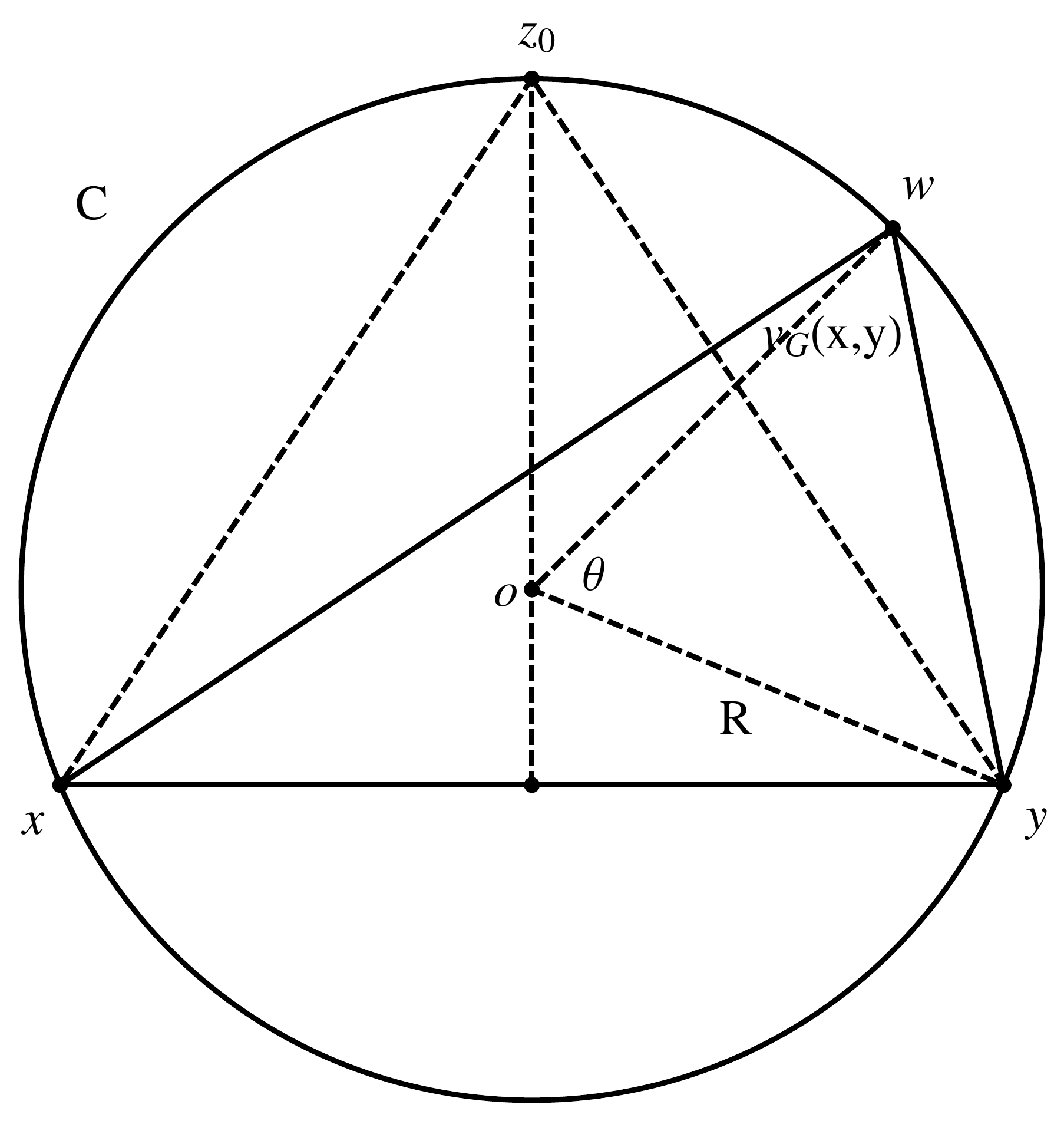}
\caption{Proof of Theorem \ref{3.1}}\label{fig3.1}
\end{center}
\end{figure}

\begin{figure}[h]

\begin{center}
\includegraphics[width=7.5cm]{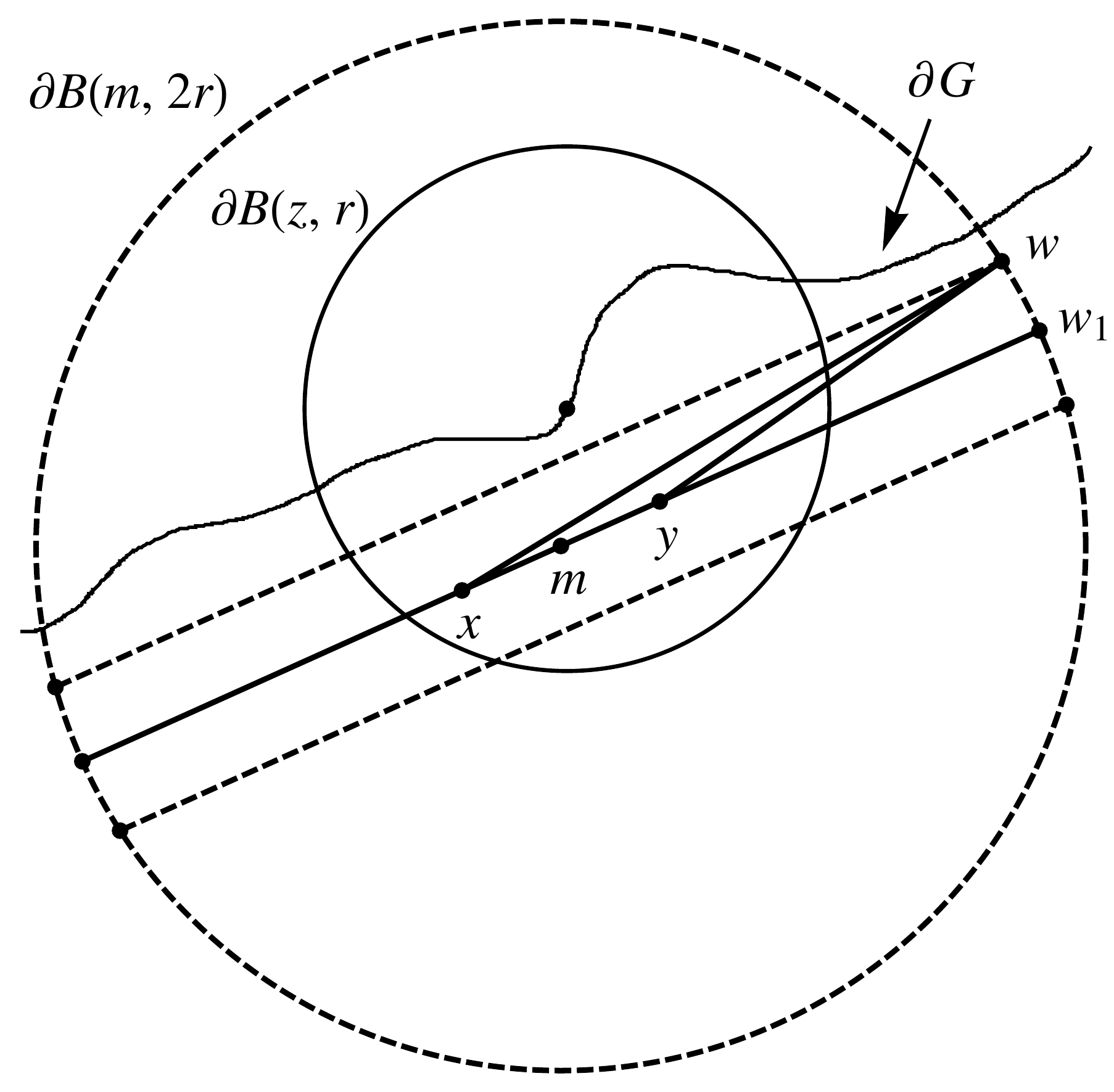}
\caption{Proof of Theorem \ref{6.1}}\label{fig2}
\end{center}
\end{figure}

In general it is not true that $v_G$ has a lower bound in terms of $s_G\,.$ For instance, this fails for $G=\B\setminus\{0\}\,,$ \cite[Remark 2.18]{hvw}.
The nonlinearity condition in the next theorem is similar to the thickness condition in \cite{vvw}, and it ensures a lower bound for $v_G$ in terms of $s_G\,.$ For the case $n=2$ an example of a domain satisfying the nonlinearity condition is the snowflake domain.

\begin{defn}
Suppose that $G\subset\R\,$ is a domain.
 We say that $\partial{G}\,,$   satisfies the nonlinearity condition, if there exists $\delta\in (0,1)\,,$ such that for every $z\in\partial{G}$ and for every $r\in (0, d(G))$ and for every line $L$ with $L\cap B(z,r)\neq \emptyset\,,$
there exists $w\in B(z,r)\cap\partial{G}\setminus \bigcup_{y\in L}B(y,\delta r)\,.$ 

\end{defn}

\begin{thm}\label{6.1}      
Let $G\subset\R$ be a domain such that $\partial{G}$ satisfies the nonlinearity condition.
If $x, y\in G$ and $s_G(x,y)<1$ then
\[
v_G(x,y)>\arctan\left(\frac{\delta}{6} s_G(x,y)\right)\,.
\]
\end{thm}

\begin{proof}
Fix $x, y\in G\,.$ We may assume that $d(x)\leq d(y)\,.$  Choose $z_0\in\partial{G}$ such that $|x-z_0|=d(x)\,.$ Let $r=d(x)+|x-y|\,.$ Then $B(z_0,r)\subset B(m,t)\,,$ $m=(x+y)/2$ for $t=2r\,.$ By the nonlinearity condition as we see in Figure \ref{fig2},
$v_G(x,y)\geq \measuredangle (x,w,y)=\alpha\,,$ $w=m+te^{i\theta}(\frac{y-x}{|y-x|})\,,$ $\theta=\arcsin{\frac{\delta r}{t}}\,.$ Writing $w_1=m+t\frac{y-x}{|y-x|}\,,$ $\beta=\measuredangle (w,y,w_1)$ and $\gamma=\measuredangle (w,x,w_1)$
we see that $\tan{\beta}=\frac{\delta r}{\sqrt{4r^2-{\delta}^2 r^2}-|x-y|/2}$ and $\tan{\gamma}=\frac{\delta r}{\sqrt{4r^2-{\delta}^2 r^2}+|x-y|/2}$ and hence

$$\tan{\alpha}=\tan\left(\beta-\gamma\right)=\frac{\delta r |x-y|}{4r^2-|x-y|^2/4}\,.$$

Therefore
\allowdisplaybreaks\begin{align*}
v_G(x,y)\geq \alpha &= \arctan{\frac{\delta r |x-y|}{4r^2-|x-y|^2/4}}\\
&= \arctan{\frac{\delta (d(x)+|x-y|) |x-y|}{4(d(x)+|x-y|)^2-|x-y|^2/4}}\\
&= \arctan{\frac{\delta (1+|x-y|/d(x)) |x-y|/d(x)}{4(1+|x-y|/d(x))^2-(|x-y|/2d(x))^2}}\,.
\end{align*}
Then $s_G(x,y)\leq \frac{|x-y|}{2d(x)}\,.$ A simple calculation shows that the function $f(t)=\frac{(1+t)t}{4(1+t)^2-(t/2)^2}$ is increasing for $t>0\,,$ since $f'(t)=2\left(\frac{1}{(4+3t)^2}+\frac{1}{(4+5t)^2}\right)>0\,.$ On the other hand, $g(t)=f(t)/t$ is decreasing for $t>0\,.$ Hence, for $0<t\leq2\,,$ $g(t)\geq g(2)=\frac{3}{35}>\frac{1}{12}$ and $f(t)\geq \frac{1}{12}t\,.$

Therefore
\allowdisplaybreaks\begin{align*}
\arctan{\frac{\delta (1+|x-y|/d(x)) |x-y|/d(x)}{4(1+|x-y|/d(x))^2-(|x-y|/2d(x))^2}} &= \arctan(f(|x-y|/d(x))\delta)\\
\geq \arctan(f(2 s_G(x,y))\delta) 
&\geq \arctan\left(\frac{\delta }{6} s_G(x,y)\right)
\end{align*}
and the proof is complete.
\end{proof}

\begin{lem}\label{lem3.3}
Let $G\subset\mathbb{R}^n$ be a proper subdomain of $\mathbb{R}^n\,,$ $x\in G$ and $y\in B^n(x,d(x))\,.$ Then
$$\sin(v_G(x,y))\leq \sup_{w\in\partial G}\frac{|x-y|}{|x-w|}=\frac{|x-y|}{d(x)}\,.$$
\end{lem}

\begin{proof}
Fix $x\in G$ and $y\in B^n(x,d(x))\,.$ For each $w\in\partial G$ we have by elementary geometry
$$\measuredangle(x-w,y-w)\leq\theta; \quad \sin\theta=\frac{|x-y|}{|x-w|}\,.$$
Taking supremum over all  $w\in\partial G$ we obtain
\[
\sin(v_G(x,y))\leq \sup_{w\in\partial G}\frac{|x-y|}{|x-w|}=\frac{|x-y|}{d(x)}\,.\qedhere
\]
\end{proof}

\begin{thm}\label{thm3.4}
Let $G$ be a proper subdomain of $\mathbb{R}^2\,.$ For $x,y\in G\,,$
$$s_G(x,y)\leq\frac{|x-y|/d(x)}{1+\cos(v_G(x,y))+\sqrt{(|x-y|/d(x))^2-\sin^2(v_G(x,y))}}\,.$$
\end{thm}

\begin{proof}
We may assume that $d(x)\leq d(y)\,.$ We first consider the case of $\partial G\cap [x,y]\neq\emptyset\,.$ It is clear in this case that $s_G(x,y)=1$ and $v_G(x,y)=\pi\,,$ and the desired inequality holds as an equality.
Next, we assume that  $\partial G\cap [x,y]=\emptyset\,.$ Let $E$ be the interior of the envelope which defines the visual angle metric between $x$ and $y\,.$ Then $D=B^2(x,d(x))\cup B^2(y,d(x)))\cup E$ is a subdomain of $G\,.$
Let $w_0\in\partial D\cap S^1(x,d(x))\cap\partial E\,.$ By use of the law of cosine in the triangle $\bigtriangleup xyw_0$ we get
$$|x-w_0|+|w_0-y|=(1+\cos(v_G(x,y)))d(x)+\sqrt{|x-y|^2-d(x)^2\sin^2(v_G(x,y))}.$$
A simple geometric observation gives
\allowdisplaybreaks\begin{align*}
s_D(x,y)&= \frac{|x-y|}{|x-w_0|+|w_0-y|}\\
&= \frac{|x-y|/d(x)}{1+\cos(v_G(x,y))+\sqrt{(|x-y|/d(x))^2-\sin^2(v_G(x,y))}}.
\end{align*}
Then the domain monotonicity of $s-$metric yields the desired inequality $s_G(x,y)\leq s_D(x,y)\,.$
\end{proof}

\begin{rem}\label{rem3.5}
(1) If $|x-y|/d(x) >1\,,$ then the square root in Theorem \ref{thm3.4} is clearly well-defined. In the case
$|x-y|/d(x)\leq 1$ it follows from Lemma \ref{lem3.3} that the square root is well-defined, too.

(2) The inequalities in Theorem \ref{thm3.4} are sharp in the following sense: If $v_G(x,y)=0\,,$ then $s_G(x,y)\leq {|x-y|}/{(|x-y|+2d(x))}$ which together with Lemma \ref{1.1}  actually gives  
\[
s_G(x,y)= {|x-y|}/{(|x-y|+2d(x))}\,.
\]
If $s_G(x,y)=1\,,$ then the inequality actually gives $v_G(x,y)=\pi\,.$
\end{rem}

\begin{defn}
Let $\delta\in (0,1/2)\,.$ We say that a domain $G\subset \Rn$ satisfies condition $H(\delta)$ if for every $z\in\partial{G}$ and all $r\in (0,d(G)/2)$ there exists $w\in B^n(z,r)\cap (\mathbb{R}^n\setminus G)$ such that $B^n(w,\delta r)\subset B^n(z,r)\cap (\mathbb{R}^n\setminus G)\,.$
\end{defn}
Note that the condition $H(\delta)$ excludes domains whose boundaries have zero angle cusps directed into the domain. For instance the domain $B^2\setminus{[0,1]}$ does not satisfy the condition $H(\delta)\,.$ A similar condition has been studied also in \cite{mv} and \cite{klv} and sometimes this condition is referred to as the porosity condition.
For instance, domains with smooth boundaries are in the class $H(\delta)\,.$

\begin{figure}[h]\label{fig3.2}
\begin{center}
     \includegraphics[width=8.5cm]{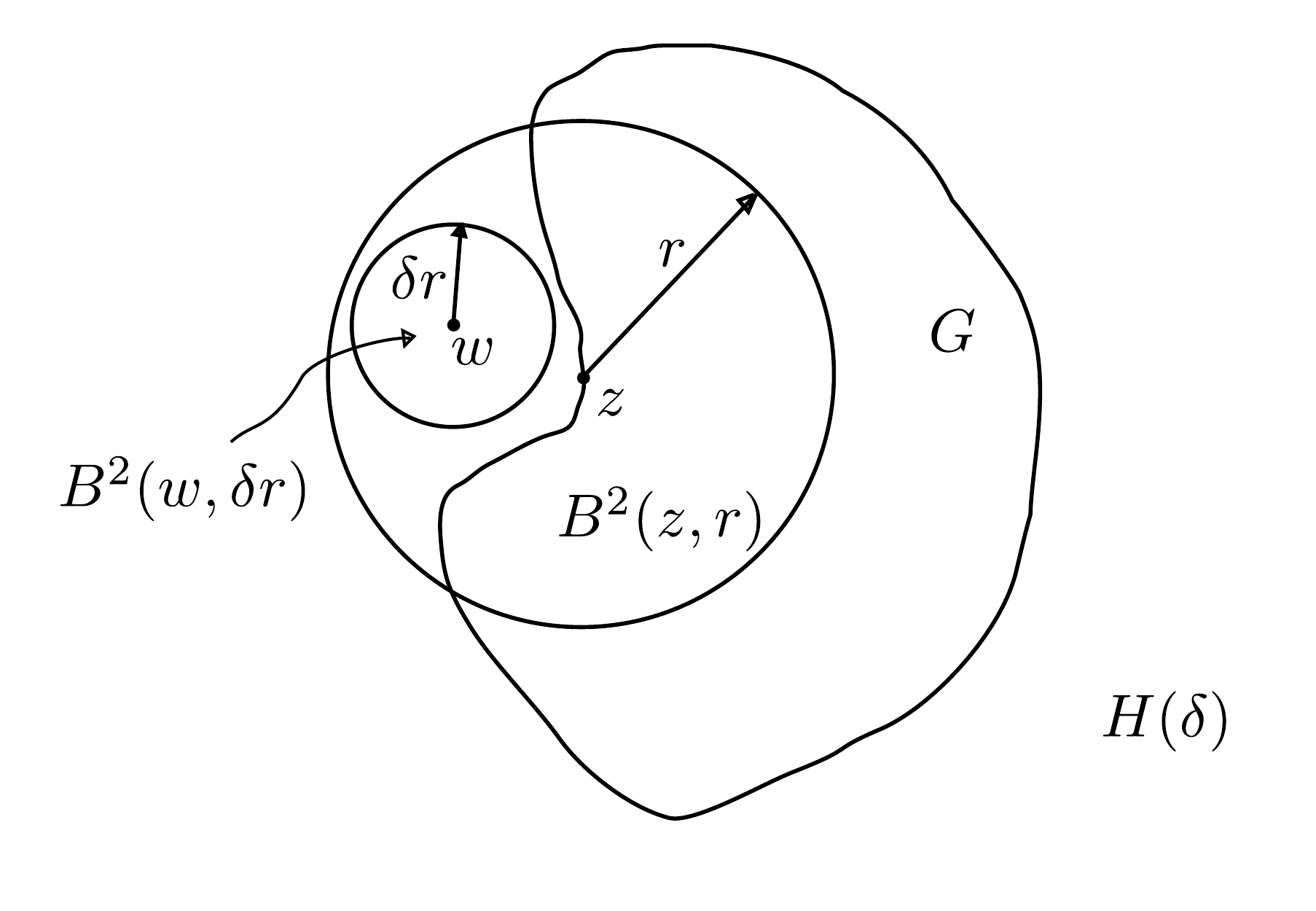}
\caption{ Condition $H(\delta)$ }
\end{center}
    \end{figure}

\begin{thm} \label{myHdthm}
Let $G\subset \R$ be a domain satisfying the condition $H(\delta)\,.$ Then for all $x, y\in G$ we have
\[
\sin{v_G(x,y)}\geq \frac{\delta}{2}j^*_G(x,y)\,.
\]
\end{thm}

\begin{proof}
Fix $x, y\in G\,.$ By symmetry we may suppose that $d(x)\leq d(y)\,.$ Denote $r=d(x)$ and choose a point $z\in\partial{G}$ such that $r=|x-z|\,.$ By the condition $H(\delta)$ there exists $w\in \mathbb{R}^2\setminus G$ such that
$B^2(w,\delta r)\subset B^2(z,r)\cap (\mathbb{R}^2\setminus G)\,.$ Denote $G_1=\mathbb{R}^2\setminus \overline{B}^2(w,\delta r)\,.$ By the monotonicity of $v_G$ with respect to the domain we have
\[
v_G(x,y)\geq v_{G_1}(x,y)\,.
\]
Geometrically, $v_{G_1}(x,y)$ can be found by considering the circle through $x, y$ externally tangent to  $B^2(w,\delta r)\,.$ Suppose this circle is $B^2(\tilde{c},\tilde{R})\,.$ In order to find a lower bound for $v_{G_1}$ we need an upper bound for $\tilde{R}\,.$ By elementary geometry $\tilde{R}\leq R$ where $B^2(c,R)$ corresponds to the case when $y=y_1=x+\frac{x-w}{|x-w|}|x-y|\,.$ Then $|x-y_1|=|x-y|\,.$ Using the power of the point $w$ with respect to the circle $\partial{B^2(c,R)}$ we have
\[
\delta r(\delta r+2R)= |x-w||y_1-w|=|x-w|(|x-w|+|x-y|)
\]
and hence
\[
2R=\frac{|x-w|}{\delta r}(|x-w|+|x-y|)-\delta r.
\]
In the same way as in the proof of Theorem \ref{3.1} we utilize the law of Sine to obtain
\[
R=\frac{|x-y_1|}{2\sin{v_{G_1}(x,y_1)}}=\frac{|x-y|}{2\sin{v_{G_1}(x,y_1)}}\,.
\]
Observing that $|x-w|\leq |x-z|+|z-w|\leq d(x)+(1-\delta)d(x)$ we have
\allowdisplaybreaks\begin{align*}
\sin{v_{G_1}(x,y)} \geq{} & \frac{|x-y|}{\frac{2-\delta}{\delta}(|x-w|+|x-y|)-\delta r}\\
\geq{} & \frac{|x-y|}{\frac{2-\delta}{\delta}((2-\delta)d(x)+|x-y|)-\delta d(x)}\\
= {}& \frac{\delta}{2-\delta}\cdot \frac{t}{t+\frac{4-4\delta}{2-\delta}}\\
\geq{} & \frac{\delta}{2-\delta}\cdot \frac{t}{t+2}=\frac{\delta}{2-\delta}\cdot \frac{e^{j_G(x,y)}-1}{e^{j_G(x,y)}+1} \geq{} \frac{\delta}{2} j^*_G(x,y)
\end{align*}
where $t=\frac{|x-y|}{d(x)}\,.$
\end{proof}

\section{Lipschitz conditions}\label{section 4}

One of the main reasons to study metrics in Geometric Function Theory is the distortion theory of mappings: the study how far  a map transforms two given points. In this section we will study the triangular ratio metric and other aforementioned metrics from this point of view. We start our discussion with the following result of F. W. Gehring and B. G. Osgood \cite{go}. We assume that the reader is familiar with the basic facts of the theory of $K-$quasiconformal/$K-$quasiregular maps  \cite{v1}, \cite{vu}. In particular,  we follow V\"ais\"al\"a's definition of  $K-$quasiconformality \cite[p. 42]{v1}.

\begin{thm}
Let $f:G\to G'$ be a $K-$quasiconformal homeomorphism between domains $G,G'\subset \mathbb{R}^n\,.$ Then there exists a constant $c=c(n,K)$ depending only on $n$ and $K$ such that for all $x,y\in G$
\[
k_{G'}(f(x)f(y))\leq \max\{k_G(x,y)^{\alpha},k_G(x,y)\}, \, \alpha=K^{1/(1-n)}\,.
\]
\end{thm}

It is a natural question that whether a similar result holds for the metrics considered in this paper. Some of these questions have already been studied elsewhere \cite{chkv,hvw}. Before proceeding we mention a few well-known cases where the above result can be refined.

\begin{lem}
Let $f:G\to G'=fG$ be a M\"obius transformation where $G,G'\subset \mathbb{R}^n$ are domains. Then
$$  (1) \quad j_G(x,y)/2 \le j_{G'}(f(x),f(y)) \le 2 j_G(x,y)\,,$$
$$  (2) \quad k_G(x,y)/2 \le k_{G'}(f(x),f(y)) \le 2 k_G(x,y)\,,$$
for all $x,y \in G \,.$
\end{lem}

\begin{proof} See \cite[Corollary 2.5]{gp} and \cite[proof of Theorem 4]{go}.
\end{proof}

\begin{lem} Let $f:G \to G'$ be a conformal map between domains $G, G' \subset  \mathbb{R}^2\,.$
Then
$$ \quad k_G(x,y)/4 \le k_{G'}(f(x),f(y)) \le 4 k_G(x,y)$$
for all $x,y \in G \,.$
\end{lem}

\begin{proof} See \cite[Proposition 1.6]{kvz}.
\end{proof}

These results may be refined further for instance if $G=G' = \mathbb{B}^n $ as shown in \cite{kvz} or if  $G=G' = \mathbb{R}^n \setminus \{ 0\} \,.$
Here our goal is to study the extent to which these results have counterparts for the triangular ratio metric.

V\"ais\"al\"a \cite{v2} has proved that an $L-$bilipschitz map with respect to the quasihyperbolic metric is a quasiconformal map with the linear dilatation $4L^2\,.$ Motivated partly by his work we consider bilipschitz maps with respect to the triangular ratio metric, and our result gives a refined upper bound $L^2$ of the linear dilatation in the case of Euclidean spaces.

\begin{thm}
Let $G\subsetneq \Rn$ be a domain and let $f:G\rightarrow fG\subset \Rn$ be a sense-preserving homeomorphism, satisfying $L$-bilipschitz condition with respect to the triangular ratio metric, i.e.
\[
 s_G(x,y)/L\leq s_{fG}(f(x),f(y))\leq L s_G(x,y)\,,
 \]
 holds for all $x, y\in G\,.$ Then $f$ is quasiconformal with the linear dilatation $H(f)\leq L^2\,.$
\end{thm}

\begin{proof}
If $x, y\in G$ satisfy $|x-y|<\min\{d(x),d(y)\}$ and $w\in\partial G$ with $d(x)=|x-w|$ is a point, then it is easy to see that
\[
s_G(x,y)\geq \frac{|x-y|}{|x-w|+|w-y|}\geq \frac{|x-y|}{2\min\{d(x),d(y)\}+|x-y|}\,,
\]
and
\[
s_G(x,y)\leq \frac{|x-y|}{d(x)+d(y)}\leq \frac{|x-y|}{2\min\{d(x),d(y)\}-|x-y|}\,,
\]
from which we conclude that
\beq\label{3}
\frac{2\min\{d(x),d(y)\}}{{1}/{s_G(x,y)}+1}\leq |x-y|\leq \frac{2\min\{d(x),d(y)\}}{{1}/{s_G(x,y)}-1}\,.
\eeq
For an arbitrary point $z\in{G}\,,$ let $x,y\in G$ with $|x-z|=|y-z|=r$ where $r$ is small enough such that the following argument is meaningful, i.e. all the terms are positive. Let
$$A(x,y,z)=\frac{\min\{d(f(x)),\,d(f(z))\}}{\min\{d(f(y)),\,d(f(z))\}}\,,$$
which tends to 1 as $x,y$ tend to $z\,.$ Then by the estimate \eqref{3} we get
\allowdisplaybreaks\begin{align*}
\frac{|f(x)-f(z)|}{|f(y)-f(z)|} &\leq A(x,y,z)\frac{{1}/{s_{fG}(f(y),f(z))}+1}{{1}/{s_{fG}(f(x),f(z))}-1}\\
& \leq A(x,y,z)\frac{{L}/s_{G}(y,z)+1}{{1}/(L s_{G}(x,z))-1}\\
& \leq A(x,y,z)\frac{{L}/({|y-z|}/(2\min\{d(y),d(z)\}+|y-z|))+1}{{1}/(L {|x-z|}/{(2\min\{d(x),d(z)\}-|x-z|)})-1}\\
&= A(x,y,z)\frac{2L^2\min\{d(y),d(z)\}+(L^2+L)|y-z|}{2\min\{d(x),d(z)\}-(L+1)|x-z|}\frac{|x-z|}{|y-z|}\\
& \rightarrow   L^2\,,
\end{align*}
when $r=|x-z|=|y-z|\rightarrow 0\,.$
Hence
\[
H(f,z)=\limsup_{|x-z|=|y-z|=r\rightarrow 0^+}{\frac{|f(x)-f(z)|}{|f(y)-f(z)|}}\leq L^2 \,.\qedhere
\]
\end{proof}

\begin{cor}
Let $G\subset \Rn$ be a domain and let $f:G\rightarrow fG\subset \Rn$ be a sense-preserving homeomorphism, satisfying $L$-bilipschitz condition with respect to the distance ratio metric or quasihyperbolic metric. Then $f$ is quasiconformal with linear dilatation $H(f)\leq L^2\,.$
\end{cor}

\begin{proof}
 By Lemma \ref{1.1}, $j^*_G(x,y)\leq s_G(x,y)\leq \frac{e^{j_G(x,y)}-1}{2}$ for all $x, y\in G\,.$ It follows that for arbitrary $\varepsilon >0\,,$ there exists $\delta >0$ such that for all $x, y\in G$ satisfying $j_G(x,y)< \delta$ we have that
\[
\frac{j_G(x,y)}{2(1+\varepsilon)}\leq s_G(x,y)\leq \frac{1+\varepsilon}{2}j_G(x,y)\,.
\]
For an $L$-bilipschitz mapping with respect to $j$-metric, we choose $x, y\in G$ such that $j_G(x,y)<\frac{\delta}{L}\,.$ Then
\begin{eqnarray*}
s_{fG}(f(x),f(y))\leq \frac{1+\varepsilon}{2}j_{fG}(f(x),f(y))&\leq & \frac{L(1+\varepsilon)}{2}j_G(x, y)\\
&\leq & L (1+\varepsilon)^2 s_G(x,y)\,.
\end{eqnarray*}
Similarly, we also have
\[
s_{fG}(f(x),f(y))\geq \frac{j_{fG}(f(x),f(y))}{2(1+\varepsilon)}\geq \frac{j_G(x, y)}{2L (1+\varepsilon)}\geq \frac{s_G(x, y)}{L (1+\varepsilon)^2}\,.
\]
Hence an $L-$bilipschitz mapping with respect to $j-$metric is in fact locally $L(1+\varepsilon)^2-$bilipschitz with respect to $s-$metric, from which we get that the mapping is quasiconformal with linear dilatation $H(f)\leq L^2(1+\varepsilon)^4\,.$ Since $\varepsilon$ is arbitrary, we conclude that the mapping is actually quasiconformal with linear dilatation $H(f)\leq L^2\,.$

Because $0<\lambda<1\,,$ $x\in G$ and $y\in B^n(x,\lambda d(x))\,,$ we have by \cite[Lemma 3.7]{vu} that $$j_G(x,y)\leq k_G(x,y)\leq j_G(x,y)/(1-\lambda)\, .$$ The same argument applies to $L-$bilipschitz mapping in $k-$metric, i.e. an $L-$bilipschitz mapping in $k-$metric is a quasiconformal mapping with linear dilatation $H(f)\leq L^2\,.$
\end{proof}

\begin{cor}
Let $G\subset \Rn$  be a domain and let $f:G\rightarrow fG\subset \Rn$ be a sense-preserving isometry with respect to the triangular ratio metric, distance ratio metric, or quasihyperbolic metric. Then $f$ is a conformal mapping. In particular, for $n\geq 3$ the mapping $f$ is the restriction of a M\"obius map.
\end{cor}

\begin{proof}
The result follows from the fact  that a $1-$quasiconformal mapping is conformal and Liouville's theorem in higher dimensions.
\end{proof}

P. H\"ast\"o \cite{hasto} has considered the isometries of the quasihyperbolic metric on plane domains and proved that, except for the trivial case of a half-plane where the quasihyperbolic metric coincides with the hyperbolic metric, the isometries are exactly the similarity mappings. Note that an additional condition of $C^3$ smoothness of the boundary of the domain is needed.

\begin{thm}\label{smobius}
Let $f:\mathbb{B}^n\rightarrow G\,,$ $G\in \{\mathbb{B}^n,\, \mathbb{H}^n \}$ be a M\"obius transformation.
Then for $x,y\in \mathbb{B}^n$ we have
\[
s_{G}(f(x),f(y))\leq \frac{2s_{\mathbb{B}^n}(x,y)}{1+s^2_{\mathbb{B}^n}(x,y)}\,.
\]
\end{thm}

\begin{proof}
For $G=\mathbb{H}^n$, by \eqref{mysH}, \cite[(2.21)]{vu} and Lemma \ref{lem2.12}, we have for all $x, y\in\mathbb{B}^n\,,$
\begin{eqnarray*}
s_{\mathbb{H}^n}(f(x),f(y))&=& \th \left( \frac{\rho_{\mathbb{H}^n}(f(x),f(y))}{2} \right) \\
&=& \th  \frac{\rho_{\mathbb{B}^n}(x,y)}{2} \\
&=& \frac{2\th\frac{\rho_{\mathbb{B}^n}(x,y)}{4}}{1+\th ^2\frac{\rho_{\mathbb{B}^n}(x,y)}{4}}\\
&\leq & \frac{2s_{\mathbb{B}^n}(x,y)}{1+s^2_{\mathbb{B}^n}(x,y)}\,.
\end{eqnarray*}
Similarly, for $G=\mathbb{B}^n\,,$ by Lemma \ref{lem2.12} and  \cite[(2.20)]{vu}, we have for all $x, y\in\mathbb{B}^n\,,$
\begin{eqnarray*}
s_{\mathbb{B}^n}(f(x),f(y))&\leq& \th \left( \frac{\rho_{\mathbb{B}^n}(f(x),f(y))}{2} \right) \\
&=& \th  \frac{\rho_{\mathbb{B}^n}(x,y)}{2} \\
&=& \frac{2\th\frac{\rho_{\mathbb{B}^n}(x,y)}{4}}{1+\th ^2\frac{\rho_{\mathbb{B}^n}(x,y)}{4}}\\
&\leq & \frac{2s_{\mathbb{B}^n}(x,y)}{1+s^2_{\mathbb{B}^n}(x,y)}\,.  
\end{eqnarray*}
\end{proof}
%
%

\begin{thm}
\begin{enumerate}

\item
Let $f:\mathbb{B}^n\rightarrow \mathbb{H}^n$ be a M\"obius transformation.
Then for $x,y\in \mathbb{B}^n$ we have
\[
p_{\mathbb{B}^n}(x,y) \leq p_{\mathbb{H}^n}(f(x),f(y))\leq \frac{2p_{\mathbb{B}^n}(x,y)}{1+p^2_{\mathbb{B}^n}(x,y)}\,.
\]

\item
Let $f:\mathbb{B}^n\rightarrow \mathbb{B}^n$ be a M\"obius transformation.
Then for $x,y\in \mathbb{B}^n$ we have
\[
\frac{p_{\mathbb{B}^n}(x,y)}{1+\sqrt{1-p^2_{\mathbb{B}^n}(x,y)}} \leq p_{\mathbb{B}^n}(f(x),f(y))\leq \frac{2p_{\mathbb{B}^n}(x,y)}{1+p^2_{\mathbb{B}^n}(x,y)}\,.
\]

\item
Let $f:\mathbb{H}^n\rightarrow \mathbb{B}^n$ be a M\"obius transformation.
Then for $x,y\in \mathbb{H}^n$ we have
\[
\frac{p_{\mathbb{B}^n}(x,y)}{1+\sqrt{1-p^2_{\mathbb{B}^n}(x,y)}} \leq p_{\mathbb{H}^n}(f(x),f(y))\leq \frac{2p_{\mathbb{B}^n}(x,y)}{1+p^2_{\mathbb{B}^n}(x,y)}\,.
\]
\end{enumerate}

\end{thm}

\begin{proof}
(1) For the second inequality by \eqref{mysH}, Theorem \ref{smobius}, and Lemma \ref{lem2.12} we have for all $x, y\in\mathbb{B}^n\,,$
\[
p_{\mathbb{H}^n}(f(x),f(y))=s_{\mathbb{H}^n}(f(x),f(y))\leq  \frac{2s_{\mathbb{B}^n}(x,y)}{1+s^2_{\mathbb{B}^n}(x,y)}
\leq \frac{2p_{\mathbb{B}^n}(x,y)}{1+p^2_{\mathbb{B}^n}(x,y)}\,.
\]
For the first inequality we have by Lemma \ref{lem2.12},
\[
p_{\mathbb{H}^n}(f(x),f(y))= \th \left( \frac{\rho_{\mathbb{H}^n}(f(x),f(y))}{2}\right)=\th\frac{\rho_{\mathbb{B}^n}(x,y)}{2} \ge p_{\mathbb{B}^n}(x,y)\,.
\]
\medskip

(2) By Lemma \ref{lem2.12} and \cite[(2.20)]{vu}
\begin{eqnarray*}
p_{\mathbb{B}^n}(f(x),f(y)) &\leq & \th\frac{\rho_{\mathbb{B}^n}(f(x),f(y))}{2}\\
&=& \frac{2\th\frac{\rho_{\mathbb{B}^n}(x,y)}{4}}{1+\th ^2\frac{\rho_{\mathbb{B}^n}(x,y)}{4}}\\
&\leq &\frac{2p_{\mathbb{B}^n}(x,y)}{1+p^2_{\mathbb{B}^n}(x,y)}\,.
\end{eqnarray*}

For the first inequality we have again by Lemma \ref{lem2.12},
\begin{eqnarray*}
p_{\mathbb{B}^n}(f(x),f(y)) &\geq &\th\frac{\rho_{\mathbb{B}^n}(f(x),f(y))}{4}\\
&= &\th\frac{\rho_{\mathbb{B}^n}(x,y)}{4}\\
&\geq & \frac{p_{\mathbb{B}^n}(x,y)}{1+\sqrt{1-p^2_{\mathbb{B}^n}(x,y)}}\,.
\end{eqnarray*}

(3) By Lemma \ref{lem2.12},
\begin{eqnarray*}
p_{\mathbb{B}^n}(f(x),f(y)) &\leq & \th\frac{\rho_{\mathbb{B}^n}(f(x),f(y))}{2}\\
&=& \frac{2\th\frac{\rho_{\mathbb{H}^n}(x,y)}{4}}{1+\th ^2\frac{\rho_{\mathbb{H}^n}(x,y)}{4}}\\
&\leq &\frac{2p_{\mathbb{H}^n}(x,y)}{1+p^2_{\mathbb{H}^n}(x,y)}\,.
\end{eqnarray*}

For the first inequality we have again by Lemma \ref{lem2.12},
\begin{eqnarray*}
p_{\mathbb{B}^n}(f(x),f(y)) &\geq &\th\frac{\rho_{\mathbb{B}^n}(f(x),f(y))}{4}\\
&= &\th\frac{\rho_{\mathbb{H}^n}(x,y)}{4}\\
&\geq & \frac{p_{\mathbb{H}^n}(x,y)}{1+\sqrt{1-p^2_{\mathbb{H}^n}(x,y)}}\,.
\end{eqnarray*}
\end{proof}

\begin{thm}\label{1.2i}

Let $f:\mathbb{B}^n\rightarrow G\,,$ $G\in \{\mathbb{B}^n,\, \mathbb{H}^n \}\,,$ be a $K-$quasiregular mapping.
Then for $x,y\in \mathbb{B}^n$ we have
\[
s_{G}(f(x),f(y))\leq \lambda_n^{1-\alpha} \left(\frac{2s_{\mathbb{B}^n}(x,y)}{1+s^2_{\mathbb{B}^n}(x,y)}\right)^{\alpha},~ \alpha=K^{1/(1-n)}\,,
\]
where $\lambda_n \in [4,2 e^{n-1}), \lambda_2=4,$ is the Gr\"otzsch ring constant depending only on $n$ (\cite[Lemma 7.22]{vu}).
\end{thm}

\begin{proof}
For $G=\mathbb{B}^n\,,$ by Lemma \ref{lem2.12} and  \cite[Theorem 5.4]{chkv}, we have for all $x, y\in\mathbb{B}^n\,,$
\begin{eqnarray*}
s_{\mathbb{B}^n}(f(x),f(y))&\leq& \th \left( \frac{\rho_{\mathbb{B}^n}(f(x),f(y))}{2} \right) \\
&\leq& \lambda_n^{1-\alpha}\left(\th  \frac{\rho_{\mathbb{B}^n}(x,y)}{2} \right)^{\alpha} \\
&=& \lambda_n^{1-\alpha}\left(\frac{2\th\frac{\rho_{\mathbb{B}^n}(x,y)}{4}}{1+\th ^2\frac{\rho_{\mathbb{B}^n}(x,y)}{4}}\right)^{\alpha}\\
&\leq & \lambda_n^{1-\alpha} \left(\frac{2s_{\mathbb{B}^n}(x,y)}{1+s^2_{\mathbb{B}^n}(x,y)}\right)^{\alpha}\,.
\end{eqnarray*}
Similarly, for $G=\mathbb{H}^n$ by \eqref{mysH}, Lemma \ref{lem2.12}, and  \cite[5.4]{chkv}, we have for all $x, y\in\mathbb{B}^n\,,$
\begin{eqnarray*}
s_{\mathbb{H}^n}(f(x),f(y))&=& \th \left( \frac{\rho_{\mathbb{H}^n}(f(x),f(y))}{2} \right) \\
&\leq& \lambda_n^{1-\alpha}\left(\th  \frac{\rho_{\mathbb{B}^n}(x,y)}{2} \right)^{\alpha} \\
&=& \lambda_n^{1-\alpha}\left(\frac{2\th\frac{\rho_{\mathbb{B}^n}(x,y)}{4}}{1+\th ^2\frac{\rho_{\mathbb{B}^n}(x,y)}{4}}\right)^{\alpha}\\
&\leq & \lambda_n^{1-\alpha} \left(\frac{2s_{\mathbb{B}^n}(x,y)}{1+s^2_{\mathbb{B}^n}(x,y)}\right)^{\alpha}\,.
\end{eqnarray*}
\end{proof}

\bigskip

{\bf Acknowledgements.}
The first author is supported by Department of Mathematics and Statistics, University of Turku, and the third author by the Academy of Finland project 268009. The authors are indebted to the referee for several valuable corrections.



\end{document}